\documentclass[12pt,a4paper]{article}
\usepackage{amsmath,amssymb,centernot, amsthm}
\usepackage{mathtools}

\newtheorem{thm}{Theorem}[section]
\newtheorem{lem}[thm]{Lemma}

\newtheorem{defi}[thm]{Definition}
\newtheorem{re}[thm]{Result}
\newtheorem{question}[thm]{Question}

\newtheorem{example}[thm]{Example}

\newtheorem{remark}[thm]{Remark}

\DeclareMathOperator{\ord}{ {\rm ord} }
\DeclareMathOperator{\Tr}{ {\rm Tr} }
\DeclareMathOperator{\wt}{ {\rm wt} }

\begin{document} 
\title{Non-projective cyclic codes whose check polynomial contains two zeros} 
\author{Tai Do Duc\\ Division of Mathematical Sciences\\
School of Physical \& Mathematical Sciences\\
Nanyang Technological University\\
Singapore 637371\\
Republic of Singapore}
\maketitle

 \begin{abstract}
 Let $n\geq 3$ be a positive integer and let $\mathbb{F}_{q^k}$ be the splitting field of $x^n-1$. By $\gamma$ we denote a primitive element of $\mathbb{F}_{q^k}$. Let $C$ be a cyclic code of length $n$ whose check polynomial contains two zeros $\gamma^d$ and $\gamma^{d+D}$, where $de \mid (q-1)$, $e>1$ and $D=(q^k-1)/e$. This family of cyclic codes is not projective. The authors in \cite{Ding, Ma, Wang, Xiong} study the weight distribution of these codes for certain parameters. In this paper, we prove that these codes are never two-weight codes.
 \end{abstract}


\section{Introduction}

A linear code is called \textbf{projective} if its dual code has weight at least $3$. 
We call a linear code \textbf{non-projective} if its dual code contains a word of weight at most $2$.
 A cyclic code is irreducible if its check polynomial is irreducible. More details about cyclic codes can be found in \cite{Lint}. 
The class of two-weight cyclic codes has been studied intensively by many authors \cite{Ding, Feng, Ma, SW, Vega1, Vega2, Wang, Xiong}.
 
\medskip

Two-weight irreducible cyclic codes were completely classified by Schmidt and White, see \cite{SW}. 
They gave necessary and sufficient conditions for the existence of these codes. 
Moreover, the nonzero weights are also explicitly described.
It remains of interest to classify all two-weight cyclic codes which are not irreducible. 
In this direction, Wolfmann \cite{Wolfmann} proved that if a two-weight projective cyclic code is not irreducible, then it is the direct sum of two one-weight irreducible cyclic subcodes of the same dimension. 
Later, Vega \cite{Vega1} and Feng \cite{Feng} complete the classification by giving necessary and sufficient conditions for these codes to be direct sum of two one-weight irreducible cyclic subcodes of the same dimension.
Nevertheless, the non-projective case remains open. 

The authors in  \cite{Ding}, \cite{Ma}, \cite{Wang}, \cite{Xiong} studied the weight distributions of cyclic codes of various parameters. All these codes are not projective codes and not two-weight codes. The studied parameters belong to a bigger family of codes whose description was given by Feng in the concluding remarks in \cite{Feng}. It is the purpose of this paper to prove that these codes are non-projective and never two-weight.

\begin{thm}  \label{main}
Let $n\geq 3$ be a positive integer. Let $q$ be a prime power and let $\mathbb{F}_{q^k}$ be the splitting field of $x^n-1$. Let $\gamma$ denote a primitive element of $\mathbb{F}_{q^k}$. Let $C$ be the cyclic code of length $n$ over $\mathbb{F}_q$ whose check polynomial is the minimal polynomial over $\mathbb{F}_q$ containing two zeros $\gamma^d$ and $\gamma^{d+D}$ in which $\gamma$ is a primitive element of $\mathbb{F}_{q^k}$ in which
$$de \mid (q-1), \ e>1, \ D=\frac{q^k-1}{e}.$$
Then the code $C$ is non-projective and $C$ is not a two-weight code. 
\end{thm}

\medskip

\section{Structure of the Code C}

In this section, we study the structure of the code $C$ described in Theorem \ref{main} and provide necessary tools for the proof of Theorem \ref{main}. First, we fix some notations and state basic definitions of cyclic codes.

Let $m$ and $n$ be coprime integers. By $\ord_n(m)$ we denote the smallest positive integer $k$ such that $m^k\equiv 1 \pmod{n}$.

\begin{defi}
Let $h(x)$ be an irreducible divisor of $x^n-1$ over $\mathbb{F}_q$, where $(q,n)=1$. The cyclic code $W$ of length $n$ over $\mathbb{F}_q$ with check polynomial $h(x)$ is called an \textbf{irreducible} cyclic code. \\
Moreover, let $\mathbb{F}_{q^k}$ be the splitting field of $x^n-1$ over $\mathbb{F}_q$ (note that $k=\ord_n(q)$). Let $\alpha$ be a root of $f(x)$ and put $\delta=\alpha^{-1}$. By $\Tr$ we denote the trace of $\mathbb{F}_{q^k}$ over $\mathbb{F}_q$. Then the code $W$ consists of the following words.
$$c_w=\left(\Tr(w), \Tr(w\delta),\dots, \Tr(w\delta^{n-1})\right), \ w\in \mathbb{F}_{q^k}.$$
\end{defi}

\medskip

The main tools used in the proof of Theorem \ref{main} is MacWilliams identities \cite{mac} and the results by Schmidt and White \cite{SW}. While MacWilliams gives relation between the weights of a linear code, Schmidt and White give an explicit description for the weights of a two-weight irreducible cyclic codes. The following result is taken from \cite[Lemma 2.2]{mac}.

\begin{re}
Let $W$ be an $[n,m]$ linear code over $\mathbb{F}_q$. Let $W^\perp$ denote the dual code of $W$. For each $i=0,\dots,n$, let $C_i (B_i)$ denote the number of words in $W (W^\perp)$ which have weight $i$. Then
\begin{equation} \label{weight}
\sum_{i=0}^n C_i {n-i \choose v}=q^{m-v} \sum_{i=0}^n B_i {n-i \choose n-v} \ \ \text{for} \ \ v=0,1,\dots,n-1.
\end{equation}
\end{re}

Let $w_1,\dots,w_N$ be all the nonzero weights in the code $W$ and let $A_i$ be the numbers of words of weight $w_i$ in $W$. Letting $v=0,1,2$ in (\ref{weight}), we obtain the following three identities which will be useful later.

\begin{re} \label{weight relation}
Under the above notations, we have 
\begin{itemize}
\item[(1)] $\sum_{i=1}^N A_i=q^m-1$.
\item[(2)] $\sum_{i=1}^N w_iA_i=(n(q-1)-B_1)q^{m-1}$.
\item[(3)] $\sum_{i=1}^N w_i^2A_i=[n^2(q-1)^2+n(q-1)-B_1(q+2(n-1)(q-1))+2B_2]q^{m-2}$.
\end{itemize}
\end{re}

\medskip

Next, we give a description for the code $C$ in Theorem \ref{main}. From now on, we always fix a prime power $q$ and positive integers $n,k,d,e,D$ with the properties $n\geq 3$, $k=\ord_n(q)$ and
\begin{equation} \label{relation}
de \mid (q-1), \ e>1, \ D=\frac{q^k-1}{e}.
\end{equation}
Fix $\gamma$ as a primitive element of $\mathbb{F}_{q^k}$. By $C$ we denote the cyclic code of length $n$ whose check polynomial is the minimal polynomial over $\mathbb{F}_q$ containing two zeros $\gamma^d$ and $\gamma^{d+D}$.

Note that there is no integer $i$ such that $0\leq i \leq k-1$ and $d+D\equiv dq^i \pmod{q^k-1}$. Otherwise, the congruence $d+(q^k-1)/e\equiv dq^i \pmod{q^k-1}$ implies $q^i \equiv 1 \pmod{(q^k-1)/(de)}$, so $i=0$ and $D\equiv 0\pmod{q^k-1}$, impossible. Hence, the minimal polynomials (over $\mathbb{F}_q$) $h_d(x)$ and $h_D(x)$ of $\gamma^d$ and $\gamma^{d+D}$ have no common zero. These polynomials are 
$$h_d(x)=(x-\gamma^d)(x-\gamma^{dq})\cdots (x-\gamma^{dq^{h-1}}), \ \text{and}$$
$$h_D(x)=(x-\gamma^{d+D})(x-\gamma^{(d+D)q})\cdots (x-\gamma^{(d+D)q^{H-1}}),$$
where $h$ and $H$ are the smallest positive integers such that 
$$d(q^h-1)\equiv 0\pmod{\frac{q^k-1}{q-1}} \ \text{and} \ (d+D)(q^H-1)\equiv 0\pmod{\frac{q^k-1}{q-1}}.$$
As $d<q-1$, we have $h=k$. Moreover note that $(q^k-1,d+D)=d\left(\frac{q^k-1}{de}e,1+\frac{q^k-1}{de}\right)=d\left(e,1+\frac{q^k-1}{de}\right)$ divides $de$, so $(d+D,(q^k-1)/(q-1))\leq de\leq q-1$. Hence we also have $H=k$. Therefore, the polynomial 
$$h(x)=h_d(x)h_D(x)$$
is a polynomial of degree $2k$ and $C$ is an $[n,2k]$ linear code.

We have proved the following lemma.

\begin{lem} \label{form of C}
 Let $C_d$ and $C_D$ be the cyclic irreducible codes whose check polynomial are $h_d(x)$ and $h_D(x)$ described as above. Then both $C_d$ and $C_D$ have dimension $k$. Moreover, the code $C$ has dimension $2k$ with check polynomial $h(x)=h_d(x)h_D(x)$. Denote $\beta=\gamma^{-1}$. The codes $C_d, C_D$ and $C$ can be explicitly described as follows.
\begin{eqnarray*}
C_d&=& \{c_u=( \Tr(u), \Tr(u\beta^d), \dots, \Tr(u\beta^{d(n-1)}) ): u \in \mathbb{F}_{q^k}\},\\
C_D&=& \{c_v=( \Tr(v), \Tr(v\beta^{d+D}), \dots, \Tr(v\beta^{(d+D)(n-1)}) ): v \in \mathbb{F}_{q^k}\},\\
C&=& \{c_{u,v}=( \Tr(u+v), \dots, \Tr(u\beta^{d(n-1)}+v\beta^{(d+D)(n-1}) ): u,v \in \mathbb{F}_{q^k}\}.
\end{eqnarray*}
\end{lem}

The existence of the code $C$ of length $n$ implies that $\beta^{dn}=1$, so $(q^k-1)\mid dn$. As $q^k-1 \equiv 0 \pmod{n}$, there exists a divisor $\lambda$ of $d$ such that 
$$n=\lambda \frac{q^k-1}{d}.$$
By Lemma \ref{Cd and CD two weight}, both $C_d$ and $C_D$ are two-weight codes if $C$ is two-weight. For the time being, we assume the validity of this result, that is, the codes $C$, $C_d$ and $C_D$ are all two-weight codes. 

By $\wt(W)$ we denote the set of weights of the code $W$. The following results in \cite{SW} allow us to focus on two-weight codes over $\mathbb{F}_p$.

\begin{re} Put $n_1=(q^k-1)/d=n/\lambda$. The following code $C'_d$ is a two-weight code of length $n_1$ and  $\wt(C_d)=\lambda \wt(C'_d)$.
$$C'_d=\{ c'_u=(\Tr(u), \Tr(u\beta^d), \cdots, \Tr(u\beta^{d(n_1-1)})): u \in \mathbb{F}_{q^k} \}.$$
Define
$$n_2=\frac{n_1(q-1)}{(q-1,n_1)}=\frac{q^k-1}{((q^k-1)/(q-1),d)} \ \text{and} \ g=\left(\frac{q^k-1}{q-1},d\right).$$
The following code $C''_d$ is an irreducible cyclic code of length $n_2$.
$$C''_d=\{ c''_u=( \Tr(u), \Tr(u\beta^g), \cdots, \Tr(u\beta^{g(n_2-1)})): u \in \mathbb{F}_{q^k} \}.$$
Moreover, the code $C''_d$ is a two-weight code and 
\begin{equation} \label{weight of C''_d}
\wt(C''_d)=\frac{d}{g}\wt(C'_d)=\frac{d}{\lambda g}\wt(C_d).
\end{equation}
\end{re}


\begin{re} \label{code over Fp}
Let $\Tr_p$ denote the trace of $\mathbb{F}_{q^k}$ over $\mathbb{F}_p$ and let $\bar{C}_d$ denote the following irreducible cyclic code over $\mathbb{F}_p$.
$$\bar{C}_d=\{\bar{c}_u=(\Tr_p(u),\Tr_p(u\beta^g)...,\Tr_p(u\beta^{g(n_2-1)})): u \in \mathbb{F}_{q^k}\}.$$
Then the code $\bar{C}_d$ is two-weight and 
\begin{equation} \label{weight of bar(C)d}
\text{wt}(\bar{C}_d)=\frac{q(p-1)}{p(q-1)}\text{wt}(C''_d).
\end{equation}
Combining (\ref{weight of C''_d}) and (\ref{weight of bar(C)d}), we obtain
\begin{equation} \label{relation for weight of barC and Cd}
\wt(C_d)=\frac{\lambda gp(q-1)}{dq(p-1)} \wt(\bar{C}_d).
\end{equation}
\end{re}


Using Result \ref{code over Fp} and \cite[Corollary 3.2]{SW}, we can describe the two weights of $C_d$ in the following result.

\begin{re} \label{two weights of the code}
Denote
$$q=p^t, \ g=\left( \frac{q^k-1}{q-1},d \right), \ h=\text{ord}_g(p), \ s=\frac{kt}{h}.$$
The following are two weights of the code $C_d$.
\begin{equation} \label{two weights of Cd}
w_1=\frac{\lambda(q-1)p^{s\theta}(p^{s(h-\theta)}-\epsilon m)}{dq}, \ \ w_2=\frac{\lambda(q-1)p^{s\theta}(p^{s(h-\theta)}-\epsilon m+\epsilon g)}{dq},
\end{equation}
where $\epsilon=\pm 1$ and $m$ is a positive integer with following properties
\begin{itemize}
\item[(i)] $m\mid (g-1)$,
\item[(ii)] $mp^{s\theta}\equiv \epsilon \pmod{g}$, where $\epsilon=\pm 1$,
\item[(iii)] $m(g-m)=(g-1)p^{s(h-2\theta)}$,
\end{itemize}
and $\theta=\theta(g,p)$ is an integer defined by
$$\theta(g,p)=\frac{1}{p-1}\min\{S_p\left(\frac{j(p^h-1)}{g}\right): 1\leq j \leq g-1\},$$
where $S_p(x)$ denotes the sum of the $p$-digits of $x$.
\end{re}


The last result in this section is taken from \cite[Theorem 12]{Wolfmann}.

\begin{re} \label{projective cyclic code}
Let $n$ be a positive integer and let $q$ be a prime power such that $(n,q)=1$. Let $C$ be a two-weight projective cyclic code of length $n$ over $\mathbb{F}_q$. Assume that $C$ is not an irreducible code. Then $C$ is the direct sum of two one-weight irreducible cyclic subcodes of the same dimension and of the same unique nonzero weight $w_1$. Moreover, all irreducible cyclic subcodes of $C$ have the same weight $w_1$.
\end{re}


\section{Proof of Theorem \ref{main}}
\begin{lem} \label{C is not projective}
Define $f=((q^k-1)/(q-1),de)$. The number $B_2$ of words in the dual code $C^\perp$ of $C$ having weight $2$ is
\begin{equation} \label{number of words weight 2 in the dual code}
B_2=\left(\frac{\lambda f (q-1)}{de}-1\right)(q-1).
\end{equation}
Moreover, the code $C$ is not a projective code.
\end{lem}

\begin{proof}
Note that there is no word in $C^\perp$ or weight $1$, as such a word induces a nonzeoro polynomial $ax^m$, $0\leq m\leq n-1$, which contains two zeros $\gamma^d$ and $\gamma^{d+D}$, impossible. Therefore, the code $C$ is projective if and only if $B_2\neq 0$.
 
\medskip 

 The number of words in $C^\perp$ having weight $2$ is equal to the number of pairs $(a_m,b_m)\in \mathbb{F}_q^{*}\times \mathbb{F}_q$ such that $1\leq m \leq n-1$ and the polynomial $a_mx^m-b_m$ contains two zeros $\gamma^d$ and $\gamma^{d+D}$. Let $N$ be the number of integers $m$ such that $1\leq m \leq n-1$ and there exists a polynomial $x^m-c_m\in \mathbb{F}_q[x]$  which contains two zeros $\gamma^d$ and $\gamma^{d+D}$. By the linearity of $C$, we have
\begin{equation} \label{relation B2 and N}
B_2=N(q-1).
\end{equation}
Note that $x^m-c_m$ has zeros $\gamma^d$ and $\gamma^{d+D}$ if and only if $\gamma^{dm}=c_m \in \mathbb{F}_q^*$ and $\gamma^{Dm}=1$. 
Hence $(q^k-1) \mid Dm$ and $(q^k-1)/(q-1)\mid dm$. 
The first condition implies $e \mid m$. 
Put  $d'=(\frac{q^k-1}{q-1},d)$. The second condition implies $\frac{q^k-1}{(q-1)d'} \mid m$. 
Thus $m$ is divisible by the following number
$$\text{lcm}\left( e,\frac{q^k-1}{(q-1)d'} \right)=\frac{(q^k-1)e}{(q-1)d'f'},$$
where $f'=(\frac{q^k-1}{(q-1)d'},e)$. We have
$$d'f'=\left( \frac{q^k-1}{q-1},ed'\right)=\left(\frac{q^k-1}{q-1},\frac{q^k-1}{q-1}e,de\right)=\left(\frac{q^k-1}{q-1},de\right)=f.$$
Therefore, $m$ is a multiple of $\frac{(q^k-1)e}{(q-1)f}=n\frac{de}{\lambda f(q-1)}$. 
The number $N$ of integers $1 \leq m \leq n-1$ which has this property is $N=\lambda f (q-1)/(de)-1$. Combining with (\ref{relation B2 and N}), we prove (\ref{number of words weight 2 in the dual code}).

\medskip

Now, assume that $C$ is projective. We have $B_2=0$, which implies 
$$de=q-1 \ \text{and} \ \lambda=f=1.$$
By Result \ref{projective cyclic code}, the irreducible subcode $C_d$ of $C$ have a unique non-zero weight $w_1$. The identities $(1)$ and $(2)$ from Result \ref{weight relation} imply
$$w_1=\frac{n(q-1)q^{k-1}}{q^k-1}=\frac{q-1}{d}q^{k-1}.$$
Note that none of words in the dual code $C_d^\perp$ of $C_d$ has weight $1$, as $\gamma^d$ cannot be zero of any nonzero polynomial $ax^m\in \mathbb{F}_q[x]$. 
Let $C_2$ be the number of words in $C_d^\perp$ having weight $2$. 
Let $M$ be the number of integers $r$ such that $1 \leq r \leq n-1$ and there exists a polynomial $x^r-c_r \in \mathbb{F}_q[x]$ which contains a zero $\gamma^d$. 
By similar reasoning as before, we obtain $C_2=M(q-1)$ and $(q^k-1)/(q-1) \mid rd$. 
As $f=((q^k-1)/(q-1),de)=1$, we have $(q^k-1)/(q-1) \mid r$. 
The number of integers $1\leq r \leq n-1$ which is a multiple of $(q^k-1)/(q-1)=nd/(q-1)$ is $(q-1)/d-1$. Thus
\begin{equation} \label{number of weight 2 words in dual Cd}
C_2=\left( \frac{q-1}{d}-1\right)(q-1).
\end{equation}
By the identity $(3)$ from Result \ref{weight relation}, we obtain
$$(q^k-1)\left(\frac{q-1}{d}\right)^2q^k=\left(\frac{(q^k-1)(q-1)}{d}\right)^2+\frac{(q^k-1)(q-1)}{d}+2(q-1)(\frac{q-1}{d}-1),$$
which implies $(q^k-1)(q-1)/d$ divides $2(q-1)((q-1)/d-1)$. This is possible only when $k=1$ and $(q-1)/d \mid 2$. We obtain $n=(q-1)/d < 3$, a contradiction.
\end{proof}

\bigskip

Since $C_d$ and $C_D$ are subcodes of $C$, they have at most two weights. In the next lemma, we prove that they cannot be one-weight codes.

\begin{lem} \label{Cd and CD two weight}
Under the same notations as above, suppose that the code $C$ is two-weight. Then both $C_d$ and $C_D$ are two-weight codes.
\end{lem}
\begin{proof}
We prove by contradiction. Suppose that either $C_d$ or $C_D$ is one-weight. Assume that is $C_d$. Note that there is no word in the dual code of $C_d$ having weight $1$. Let $w_1=\text{wt}(C_d)$. By the equation $(2)$ of Result \ref{weight relation}, we obtain $(q^k-1)w_1=n(q-1)q^{k-1}$. Hence
\begin{equation} \label{first weight}
w_1=\mu q^{k-1}, \ \text{where} \ \mu=\frac{\lambda (q-1)}{d} \mid (q-1).
\end{equation}
Note that $w_1$ is also one weight of $C$. Next, we apply the MacWilliams identities again to find the other weight $w_2$ of $C$. Recall that $A_1$ and $A_2$ be the numbers of words in $C$ of weights $w_1$ and $w_2$. Moreover, the numbers $B_1$ and $B_2$ denote the numbers of words in $C^\perp$ of weights $1$ and $2$. Note that $B_1=0$ and the value of $B_2$ is given in (\ref{number of words weight 2 in the dual code}). By Result \ref{weight relation}, we have the following identities for the $[n,2k]$ cyclic code $C$.
\begin{itemize}
\item[(1)] $A_1+A_2=q^{2k}-1$.
\item[(2)] $A_1w_1+A_2w_2=n(q-1)q^{2k-1}$.
\item[(3)] $A_1w_1^2+A_2w_2^2=\left( n^2(q-1)^2+n(q-1)+2\left(\frac{\lambda f(q-1)}{de}-1\right)(q-1) \right) q^{2k-2}$.
\end{itemize}
As $(A_1w_1+A_2w_2)(w_1+w_2)-(A_1+A_2)w_1w_2=A_1w_1^2+A_2w_2^2$, we obtain
 \begin{equation} \label{key equation}
nq(w_1+w_2)-\frac{(q^{2k}-1)w_1w_2}{(q-1)q^{2k-2}}=n^2(q-1)+n+2\frac{\lambda f(q-1)}{de}-2. 
 \end{equation}
Note that $w_1=\mu q^{k-1}$ with $\mu \mid (q-1)$, by (\ref{first weight}). The equation (\ref{key equation}) implies that $w_2=\alpha q^{k-1}$ for some $\alpha \in \mathbb{Z}^+$. In (\ref{key equation}) using $(q^k-1)\mu/(q-1)=n$, we obtain
$$nq^k(\mu+\alpha)-n(q^k+1)\alpha=n^2(q-1)+n+2\frac{\lambda f(q-1)}{de}-2,$$
which implies $n \mid (2\lambda f(q-1)/(de)-2)$. By Lemma \ref{C is not projective}, the number $2\lambda f(q-1)/(de)-2$ is nonzero, as $B_2 \neq 0$. Thus
$$n<2\frac{\lambda f(q-1)}{de}\leq 2\lambda (q-1),$$
as $f=((q^k-1)/(q-1),de)\leq de$. Since $d\leq (q-1)/e \leq (q-1)/2$, we have 
$$2\lambda\frac{q^k-1}{q-1}\leq n=\lambda \frac{q^k-1}{d} < 2\lambda (q-1),$$
which implies $k=1$. In this case, we have $f=((q^k-1)/(q-1),de)=1$ and the inequality $n<2\lambda f (q-1)/(de)$ implies
$$\lambda \frac{q-1}{d}=n < \frac{2\lambda(q-1)}{de},$$
so $e\leq de<2$, a contradiction.

\end{proof}

\bigskip

\textbf{Proof of Theorem \ref{main}}

\begin{proof}
We prove by contradiction. Suppose that $C$ is two-weight. 
Let $w_1$ and $w_2$ denote the two nonzero weights of $C$. 
By Lemma \ref{Cd and CD two weight}, both $C_d$ and $C_D$ are also two-weight. 
The equation (\ref{key equation}) implies that $q^{2k-2} \mid w_1w_2$. 
We show that the values of $w_1$ and $w_2$ defined in (\ref{two weights of Cd}) cannot satisfy this condition.
Recall that
$$w_1=\frac{\lambda(q-1)p^{s\theta}(p^{s(h-\theta)}-\epsilon m)}{dq}, \ \ w_2=\frac{\lambda(q-1)p^{s\theta}(p^{s(h-\theta)}-\epsilon m+\epsilon g)}{dq},$$
where $\epsilon=\pm 1$ and $m$ is a positive integer with following properties
\begin{itemize}
\item[(i)] $m\mid (g-1)$,
\item[(ii)] $mp^{s\theta}\equiv \epsilon \pmod{g}$, where $\epsilon=\pm 1$,
\item[(iii)] $m(g-m)=(g-1)p^{s(h-2\theta)}$,
\end{itemize}
and $\theta=\theta(g,p)$ is defined by
$$\theta(g,p)=\frac{1}{p-1}\min\{S_p\left(\frac{j(p^h-1)}{g}\right): 1\leq j \leq g-1\}.$$
Since $q^{2k-2} \mid w_1w_2$, we have $q^{2k}=p^{2kt} \mid p^{2s\theta}(p^{s(h-\theta)}-\epsilon m)(p^{s(h-\theta)}-\epsilon m+\epsilon g)$. Note that $kt=sh$, so $p^{2s(h-\theta)}$ divides $(p^{s(h-\theta)}-\epsilon m)(p^{s(h-\theta)}-\epsilon m+\epsilon g)$. The difference between $(p^{s(h-\theta)}-\epsilon m+\epsilon g)$ and $(p^{s(h-\theta)}-\epsilon m)$ is $\epsilon g$, a divisor of $(q-1)$ and not divisible by $p$. Thus, only one of the numbers $(p^{s(h-\theta)}-\epsilon m)$ or $(p^{s(h-\theta)}+\epsilon (g-m))$ is divisible by $p^{2s(h-\theta)}$.\\ 

\underline{Case 1}. $(p^{s(h-\theta)}-\epsilon m)$ is divisible by $p^{2s(h-\theta)}$.\\
Write $m=a p^{s(h-\theta)}, a \in \mathbb{Z}^+$. By (iii), we have $g-1=ap^{s\theta}(g-m)$. Note that $m \mid (g-1)$ and $p^{s\theta} \geq p \geq 2$, so $m=g-1$ and $g=1+ap^{s\theta}$. The equation (iii) again implies $h=2\theta$. Note that $h=\text{ord}_g(p)$, so $g=1+ap^{s\theta}$ divides $p^h-1=p^{2\theta}-1$. We obtain $s=1$ and $a=1$. The condition (ii) implies $\epsilon=1$. We obtain $p^{s(h-\theta)}-\epsilon m=0$ and thus $w_1=0$, a contradiction.\\

\underline{Case 2}. $(p^{s(h-\theta)}+\epsilon (g-m))$ is divisible by $p^{2s(h-\theta)}$.\\
Write $g-m=(ap^{s(h-\theta)}-\epsilon) p^{s(h-\theta)}, a \in \mathbb{Z}^+$. By (iii), we have $$g-1=(ap^{s(h-\theta)}-\epsilon)p^{s\theta}m=mp^{sh}\left( a-\frac{\epsilon}{p^{s(h-\theta)}}\right).$$ 
Note that $g \mid (p^h-1)$ and $\theta \leq h-1$, so 
$$\left(a-\frac{\epsilon}{p^s}\right) mp^{sh} \leq g-1 <p^h.$$ 
We obtain $a=m=s=\epsilon=1$ and $g-1=p^h-p^\theta$. Replacing $m=1$ into (iii), we obtain $g-1=(p^{h-\theta}-1)p^{h-\theta}$. Thus, $h=2\theta$. The condition (ii) implies $p^{\theta}\equiv 1 \pmod{g}$, contradicting with $\text{ord}_g(p)=h=2\theta$. 
\end{proof}

\end{document}